\documentclass[12pt]{article}
\usepackage{amsmath, amsthm, amsfonts, amssymb}
\newtheorem{thm}{Theorem}[section]
\newtheorem{cor}[thm]{Corollary}
\newtheorem{lem}[thm]{Lemma}
\newtheorem{prop}[thm]{Proposition}
\newtheorem{example}[thm]{Example}
\newtheorem{remarks}[thm]{Remarks}
\newtheorem{defn}[thm]{Definition}

\numberwithin{equation}{section}

\title{\bf {$L^1$-Uniqueness of the
 Fokker-Planck equation  on a Riemannian manifold
  }}
\author{Bin Qian\thanks{School of Mathematics and Statistics, Changshu Institute of Technology    215500, Changshu, China,  China Email: binqiancn@yahoo.com.cn}
 \ \ \ \ \ \  Liming Wu\thanks{Laboratoire de Math\'ematiques Appliqu\'ees,
CNRS-UMR 6620,  Universit\'e Blaise Pascal,  63177 AUBIERE, France.
Email : Li-Ming.Wu@ math.univ-bpclermont.fr}}
\date{}

\newcommand{\dd}{\mathbb{D}}
\newcommand{\ee}{\mathbb{E}}

\newcommand{\nn}{\mathbb{N}}
\newcommand{\rr}{\mathbb{R}}
\newcommand{\pp}{\mathbb{P}}

\def\CC{\mathcal C}
\def\DD{\mathcal D}

\def\FF{\mathcal F}

\def\II{\mathcal I}

\def\LL{\mathcal L}

\def\vep{\varepsilon}
\def\<{\langle}
\def\>{\rangle}

\def\d"{^{\prime\prime}}

\def\bequ{\begin{equation}}
\def\nequ{\end{equation}}

\def\bdef{\begin{defn}}
\def\ndef{\end{defn}}

\def\bthm{\begin{thm}}
\def\nthm{\end{thm}}

\def\bprop{\begin{prop}}
\def\nprop{\end{prop}}

\def\brmk{\begin{remarks}}
\def\nrmk{\end{remarks}}

\def\bexam{\begin{example}}
\def\nexam{\end{example}}

\def\blem{\begin{lem}}
\def\nlem{\end{lem}}

\def\bcor{\begin{cor}}
\def\ncor{\end{cor}}

\def\bprf{\begin{proof}}
\def\nprf{\end{proof}}

\def\dps{\displaystyle}

\def\bdes{\begin{description}}
\def\ndes{\end{description}}

\begin{document}
\maketitle
\begin{abstract}
In this paper, we obtain a necessary and sufficient condition for
$L^{\infty}$-uniqueness of Sturm-Liouville operator $a(x)\frac
{d^2}{dx^2} + b(x) \frac d{dx} -V$ on an open interval of $\rr$,
which is equivalent to the $L^1$-uniqueness of the associated
Fokker-Planck equation. For a general elliptic operator
$\LL^V:=\Delta +b \cdot\nabla -V$ on a Riemannian manifold, we
obtain sharp sufficient conditions for the $L^1$-uniqueness of the
Fokker-Planck equation associated with $\LL^V$, via comparison with
a one-dimensional Sturm-Liouville operator. Furthermore the
$L^1$-Liouville property is derived as a direct consequence of the
$L^\infty$-uniqueness of $\LL^V$.

\end{abstract}


\section{Introduction}

On a connected non-compact Riemannian manifold $M$ without boundary,
consider the heat diffusion governed by $\LL^Vf:=\Delta f + b\cdot
\nabla f -V f$ where $f\in C_0^\infty(M)$ (the space of all real
infinitely differentiable functions with compact support), where
$\Delta, \nabla$ are respectively the Laplace-Beltrami operator and
the gradient on $M$. Here the vector field $b$ is locally
Lipschitzian and represents the macroscopic velocity  of the heat
diffusion, $V: M\to \rr^+$ is a locally bounded potential killing
the heat.

Let $u(t,x)dx$ be the heat distribution at time $t$. It satisfies
the well known Fokker-Planck equation in the distribution sense

\bequ \label{FP}
\partial_t u = (\LL^V)^* u(t,x)= \Delta u - div(ub)  -
V u, \ u(0,\cdot)\ \text{given}. \nequ

A $L^1(M,dx)$-solution to (\ref{FP}) means that $t\to
u(t)=u(t,\cdot)$ is continuous from $\rr^+$ to $L^1(M,dx)$ and
$$
\<u(t)-u(0), f\> = \int_0^t \<u(s), \LL^V f\> ds, \ \forall t\ge0, \
f\in C_0^\infty(M),
$$
where $\<f,g\>=\int_Mf(x)g(x)dx$.

 The study on this subject has a long history when $\LL^V
=\Delta$:

\bdes

\item{a)} the subject was opened by S. T. Yau \cite{Yau1, Yau2}.
Once if $M$ is complete and $1<p<+\infty$, every nonnegative
subharmonic functions in  $L^p(M,dx)$ are constant (\cite{Yau1}),
and the $L^p$-uniqueness of the above Fokker-Planck equation holds
(due to Strichwarz \cite{Str1}). In \cite{WZ2} it is proved that the
$L^p$-Liouville property for nonnegative subharmonic functions
implies the $L^p$-uniqueness of the above Fokker-Planck equation for
general $\LL^V$ instead of $\Delta$.

\item{b)} For the $L^\infty$-Liouville property, Yau \cite{Yau1}
proved that every bounded harmonic function is constant if $M$ has
nonnegative Ricci curvature. The last curvature condition is shown
to be sharp, since there are infinitely many bounded harmonic
functions on a simply connected manifold with sectional curvature
identically $-1$. The final result in this opposite direction was
obtained by Sullivan \cite{Su} and Anderson \cite{An}: on a complete
$M$ with (strongly) negative sectional  curvature they identified
the Martin boundary of $M$ as the sphere at infinity $S(\infty)$.
See Anderson-Schoen \cite{AS}, Schoen-Yau \cite{SY}  for development
of this subject.

\item{c)} For the $L^\infty$-uniqueness of (\ref{FP}) with
$\LL^V=\Delta$, Davies \cite{Da} proved that it is equivalent to the
stochastic completeness of $M$ (i.e., the Brownian motion on $M$
does not explode). Grigor'yan \cite{Gri1} found sharp volume growth
condition for the stochastic completeness of $M$.

\item{d)} The question of $L^1$-uniqueness for  (\ref{FP}) is much more delicate.
Azencott \cite{Az} and P. Li and Schoen \cite{LS} found several
counter-examples for which the $L^1$-uniqueness of (\ref{FP}) fails.
P. Li \cite{Li} found the following sharp sufficient condition for
the $L^1$-uniqueness of (\ref{FP}) (with $\LL^V=\Delta$) on a
complete Riemannian manifold : \bequ \label{PLi} Ric_x \ge -C( 1+
d(x,o)^2) \nequ where $Ric_x$ is the Ricci curvature at $x$, $C>0$
is some constant, $o$ is some fixed point and $d(x,o)$ is the
Riemannian distance. Under that condition he proved that every
nonnegative $L^1(M,dx)$-subharmonic function is constant.

\ndes

Recently the second named author and Y. P. Zhang \cite{WZ2}
introduced the $L^\infty$-uniqueness of $\LL^V$ and prove that it is
equivalent to the $L^1$-uniqueness of (\ref{FP}) and also to the
$L^1$-uniqueness of the resolvent equation:

\bequ \label{resolvent} \text{if }\ u\in L^1(M,dx)\ \text{ verifies
}\ [(\LL^V)^* -1] u=0, \ \text{ then }\ u=0. \nequ

Furthermore when $M=\rr^d$ and $V=0$, necessary and sufficient
conditions are found for the $L^1$-uniqueness of (\ref{FP}) in the
one-dimensional case ($d=1$), and sharp sufficient conditions are
obtained in the multi-dimensional case. Our main purpose of this
work is to generalize the results of \cite{WZ2}. However this is not
just a generalization, indeed the new difficulty is comparable to
that in the classical passage from the Laplacian $\Delta$ to the
Schr\"odinger operator $-\Delta +V$.

The $L^2$-uniqueness for (\ref{FP}) might seem to be the most
important and natural. This is true from the point of view of
quantum mechanics when $b=0$ (in such case it is also equivalent to
the $L^2$-uniqueness of the associated Schr\"odinger equation or the
essential self-adjointness of $-\Delta +V$). But from the point of
view of heat diffusion, the $L^1$-uniqueness is physically
meaningful and it is then  important: indeed in the heat diffusion
interpretation, $u(t,x)\ge 0$ is the energy ($=$ heat) density and
the $L^1$-norm $\int_M |u(t,x)| dx$ is the total energy in the
system at time $t$; the quantities $\int u^2(t,x) dx$ or $\int
|\nabla u(t,x)|^2 dx$, though called energy in mathematical
language, are not energy in the physics of heat diffusion.

Let us explain where comes the non-uniqueness of solutions to the
Fokker-Planck equation (\ref{FP}) from two points of view.

{\bf  1) Mathematically.} When $M$ is not complete, one can impose
 different boundary conditions on the ``boundary" $\partial M:= \bar M \backslash
 M$ (which may vary and depend on different topologies) to obtain different solutions, such as Dirichlet boundary and
 Neumann boundary etc. Even if $M$ is complete, integrability or
 growth conditions will be required to assure the uniqueness of
 solution.

 {\bf 2) Physically.} The non-uniqueness comes from the
 interchange of heat between $M$ and its ``boundary". For example
 the $L^\infty$-uniqueness of (\ref{FP}) with $\LL^V=\Delta$ is
 equivalent to the non-explosion of the Brownian Motion on $M$ (i.e. $M$
  is stochastically complete) by \cite{Da},
 which means that the heat from the interior of $M$ can not reach
 the boundary $\partial$ (the one-point compactification of $M$).
 This intuitive idea is realized on a connected open domain $M$ of
 $\rr^d$ for $\Delta -V$ and for  the Nelson's diffusions
 $\Delta - \nabla \phi\cdot \nabla$ by the second named author in
 \cite{Wu98} and \cite{Wu99}.

 There is another way of interchange of heat between $M$ and its
 ``boundary": the heat at the boundary can enter into the interior
 of $M$. Indeed for the one-dimensional Sturm-Liouville operator
 without killing potential (i.e., $V=0$) on an open interval $M$ of $\rr$, the second named author
 with Y. Zhang \cite{WZ1, WZ2} proved that the $L^1$-uniqueness of
 the associated Fokker-Planck equation is equivalent to say that
 the boundary is {\it no entrance } boundary in the classification
 of Feller, which exactly means in the probabilistic interpretation
 that the heat  at the boundary can not enter into the interior
 of $M$. This is very intuitive: if the heat at the ``boundary" can
enter into the interior of $M$, new energy can be inserted from the
``boundary" into $M$ without being perceived by the local operator
$\LL^V$, and then destroys the $L^1$-uniqueness of (\ref{FP}).

The goal of this work is to realize the last physical intuition for
general $\LL^V$. All results in this work are inspired by
probabilistic (=physical) ideas, but for a larger audience all
crucial proofs will be analytic.

This paper is organized as follows. In the next section, we
introduce some preliminaries and present characterizations and
applications of the $L^\infty$-uniqueness of $\LL^V$ to the
$L^1$-Liouville property. Section 3 is devoted to the study of one
dimensional Sturm-Liouville operators $\LL^V=a(x)\frac
{d^2}{dx^2}+b(x)\frac d{dx} -V$. We shall furnish a necessary and
sufficient condition for  the $L^\infty$-uniqueness of $\LL^V$ by
means of a new notion of {\it no entrance boundary}. A comparison
principle is derived and several examples are presented. In Section
4, we establish a sharp sufficient condition for the
$L^\infty$-uniqueness of $\LL^V$ on Riemannian manifolds by means of
comparison with a one-dimensional model. Several examples are
presented.

\section{$L^\infty$-uniqueness of pre-generator and $L^1$-Liouville property}

Throughout this paper we assume that vector filed $b$ is locally
Lipschitzian and the killing potential $V$ is nonnegative and
locally bounded (measurable of course).

\subsection{Background on $L^\infty$-uniqueness of pre-generator}

Given the operator $\LL^V$ acting on $C_0^\infty(M)$, let
$(X_t)_{0\le t\le \sigma}$ be the (stochastic) diffusion generated
$\LL = \Delta + b\cdot\nabla$, defined on $(\Omega, \FF,
(\pp_x)_{x\in M})$, where $\sigma$ is the explosion time (see
Ikeda-Watanabe \cite{IW}). Then by Feynman-Kac formula,

\bequ \label{21a} P_t^V g(x) = \ee^x 1_{t<\sigma}g(X_t)
\exp\left(-\int_0^t V(X_s) ds\right) \nequ is one semigroup
generated by $\LL^V$, i.e., for all $t\ge0$,
$$
P_t^V f - f =\int_0^t P_s^V (\LL^V f) ds, \ \forall f\in
C_0^\infty(M).
$$
But $(P_t^V)$ is not strongly continuous on $L^\infty(M, dx)$ w.r.t.
the norm $\|\cdot\|_\infty$ (indeed Lotz's theorem says that  the
generator of every strongly continuous (or $C_0-$) semigroup of
operators on $(L^\infty, \|\cdot\|_\infty)$ is bounded). So it is
impossible to define the uniqueness of $C_0-$semigroup on $L^\infty$
generated by $\LL^V$, in the norm $\|\cdot\|_\infty$. That's why we
introduce in \cite{WZ2} the topology $\CC(L^\infty, L^1)$ on
$L^\infty$ of uniform convergence over compact subsets of $L^1$. It
is proved in \cite{WZ2} that a semigroup of bounded operators on
$L^\infty$ is strongly continuous on $L^\infty$ with respect to
$\CC(L^\infty, L^1)$ if and only if $(P_t)=(Q_t^*)$, where $(Q_t)$
is a $C_0$-semigroup on $L^1$ (w.r.t. the $L^1$-norm). Now the
$L^\infty$-uniqueness of $\LL^V$ can be defined as

\begin{defn}{\rm (\cite{WZ2})}
We call that $\LL^V$ is $L^\infty$-unique,  if the closure
$\overline{\LL^V}$ of $(\LL^V, C_0^\infty(M))$ is the generator of
$(P_t^V)$ on $(L^\infty, \CC(L^\infty, L^1))$, in the graph topology
induced by $\CC(L^\infty, L^1)$.
\end{defn}

 Let
\bequ \label{radius} \lambda_0:= \lim_{t\to\infty} \frac 1t \log
\|P_t^V1\|_\infty = \inf_{t>0} \frac 1t \log \sup_{x\in M}
P_t^V1(x),\nequ i.e., $e^{\lambda_0 t}$ is the spectral radius of
$P_t^V$ in $L^\infty(M,dx)$, which is always in the spectrum of
$P_t^V$ in $L^\infty(M,dx)$. Since $V\ge0$, we have always
$\lambda_0\le0$. Recall

\bthm \label{thm21} {\rm (\cite[a particular case of Theorem
2.1]{WZ2})} The following properties are equivalent:

\bdes

\item{(i)} $\LL^V$ is $L^\infty$-unique;

\item{(ii)} for some or equivalently for all $\lambda>\lambda_0$, if $ u\in L^1(M,dx)$ verifies
$[(\LL^V)^* -\lambda] u=0$, then $u=0$;

\item{(iii)} the Fokker-Planck equation (\ref{FP}) has a unique
$L^1(M,dx)$-solution;

\item{(iv)} $(P_t^V)$ given by (\ref{21a}) is the unique
$C_0$-semigroup on $(L^\infty, \CC(L^\infty, L^1))$ such that its
generator extends $\LL^V$. \ndes

\nthm

By the theory for elliptic partial differential equations (PDE),
$$P_t^Vf(x)= \int_M p_t^V(x,y) f(y)dy$$
and it is known that if  $0\le u(0)=u(0,\cdot)\in L^1(M,dx)$,
$$
u(t,y):= (P_t^V)^* u(0) (y) = \int_M u(0,x) p_t^V(x,y) dx
$$
is the minimal nonnegative solution to (\ref{FP}).

\subsection{$L^1$-Liouville property}

At first we should understand the meaning of harmonic functions
related with $\LL^V$. When $\LL^V=\Delta$, a harmonic function $h$
(i.e., $\Delta h=0$) is a solution independent of $t$  to (\ref{FP})
(i.e., the equilibrium distribution of heat). For $\LL=\Delta
+b\cdot \nabla$, the equilibrium distribution $h$ of heat satisfies
Kolmogorov's equation
$$
\LL^* h=\Delta h - div(hb)=0.
$$
However in presence of the killing potential $V\ge0$, usually
equilibrium distribution $h$ is zero. So some further interpretation
is required. Since $p_t^V(x,y)>0, \ dy-a.e.$ for every $x\in M$, the
dimension of \bequ\label{II} \II:=\{h\in L^1(M,dx); \ (P_t^V)^* h=
e^{\lambda_0 t}h, \forall t\ge0 \} \nequ is at most one
(Perron-Frobenius theorem), and if its dimension is one, then it is
generated by some strictly positive $h_0$ such that $\int_M h_0
dx=1$ (by the theory of positive operators \cite{MN91}).

\bdef A function $h\in L^1_{loc}(M,dx)$ (the space of real locally
$dx$-integrable functions on $M$) is said to be
$(\LL^V-\lambda)^*$-harmonic where $\lambda\in \rr$, if
$$
\<h, (\LL^V -\lambda ) f\>=0, \ \forall f\in C_0^\infty(M)
$$
(recall that $\<f,g\>=\int_M fgdx$). It is said to be
$(\LL^V-\lambda)^*$-subharmonic, if

$$
\<h, (\LL^V -\lambda ) f\>\ge0, \ \forall 0\le f\in C_0^\infty(M).
$$

 \ndef

We  now state our result about the $L^1$-Liouville property.

\bthm\label{thm22} Assume that $\LL^V$ defined on $C_0^\infty(M)$ is
$L^\infty$-unique. Let $\lambda\in\rr$. Then for $h\in L^1(M,dx)$,
it is $(\LL^V-\lambda)^*$-harmonic if and only if
$$
(P_t^V)^* h= e^{\lambda t} h, \ \forall t\ge0.
$$
In particular we have the following alternatives :  \bdes \item{(a)}
If $\lambda>\lambda_0$ or $\lambda=\lambda_0$ but $dim(\II)=0$, then
every $(\LL^V-\lambda)^*$-harmonic function $h$ in $L^1(M,dx)$ is
zero.

\item{(b)} If $\lambda=\lambda_0$ and $dim(\II)=1$, then every
 $(\LL^V-\lambda)^*$-harmonic function $h$ in $L^1(M,dx)$ is $c h_0$
where $h_0$ is the strictly positive element in $\II$ such that
$\int_M h_0 dx=1$ and $c$ is a constant.

\ndes \nthm

The results above without the $L^\infty$-uniqueness of $\LL^V$ are
in general false, see Li-Schoen's Example \ref{exa41}.

 \bprf The sufficient part is obvious by differentiating on $t=0$
 (and holds true even without the $L^\infty$-uniqueness of $\LL^V$). Let us prove the
 necessity. Consider the generator $\LL^V_{(\infty)}$ of $(P_t^V)$
in $L^\infty(M,dx)$. For every $f$ belonging the domain of
definition $\dd(\LL^V_{(\infty)})$, there is a nest $(f_i)$ in
$C_0^\infty(M)$ such that
$$
f_i\to f, \ \LL^V f_i\to \LL^V_{(\infty)}f
$$
in the topology $\CC(L^\infty, L^1)$ by the assumed
$L^{\infty}$-uniqueness. Thus we obtain for all $f\in
\dd(\LL^V_{(\infty)})$,
$$
\<h, (\LL^V_{(\infty)} - \lambda) f\>=0
$$
which implies (since $P_t^V f\in \dd(\LL^V_{(\infty)})$ for all
$t\ge0$)
$$
\frac d{dt}\<e^{-\lambda t} (P_t^V)^*h, f\>=\frac d{dt} \<h,
e^{-\lambda t} P_t^V f\>=\<h, (\LL^V_{(\infty)} - \lambda)
e^{-\lambda t} P_t^V f\>=0, \ \forall t\ge0
$$
where it follows that $\<e^{-\lambda t} (P_t^V)^*h, f\>=\<h,f\>$.
Since $\dd(\LL^V_{(\infty)})$ is dense in $L^\infty(M,dx)$ with
respect to  $\CC(L^\infty, L^1))$ (\cite{WZ2}), we  get $e^{-\lambda
t} (P_t^V)^*h=h$ for all $t\ge0$.

 When $\lambda>\lambda_0$, the Liouville property in (a) is
 equivalent to the $L^\infty$-uniqueness of $\LL^V$ (\cite[Theorem
 0.2 or Theorem 2.1]{WZ2}). If $\lambda=\lambda_0$, the last part of (a) and (b) follow easily
from the previous equivalence.
 \nprf

\bexam\label{exa21} {\rm $\LL^V=\Delta$.  Assume that $\Delta$ is
$L^\infty$-unique. By Theorem \ref{thm22}, we have

\bdes \item{(i)} If $M$ is not stochastically complete, then every
integrable harmonic function $h$ is zero. Indeed by Theorem
\ref{thm22} we have $P_t ^*h=h$ where $(P_t)$ is the Brownian motion
(or heat) semigroup. Then $P_t^*|h|\ge |h|$. Since  $P_t1<1$
everywhere on $M$, we get that if $h\ne 0$ in $L^1(M,dx)$,
$$
\<1, |h|\> > \<P_t 1, |h|\> =\<1, P_t^*|h|\> \ge \<1, |h|\>
$$
which is a contradiction.

\item{(ii)} If $M$ is stochastically complete and the volume of
$M$ is infinite, then $dim(\II)=0$ and consequently every integrable
harmonic function is zero.

Indeed, if in contrary $dim(\II)=1$, i.e., $\II$ is spanned by some
nonnegative non-zero function $h_0\in L^1(M,dx)$, since
$\lambda_0=0$ by the stochastic completeness of $M$, $h_0dx$ is an
invariant probability measure of the Brownian motion semigroup
$(P_t)$, which implies that the kernel $R_1:=\int_0^\infty e^{-t}
P_tdt$ is positively recurrent (\cite[Proposition 10.1.1]{MT93}).
But for such Markov kernel $R_1$, it has no other nonnegative
invariant measure than $ch_0 dx$ (\cite[Theorem 10.0.1]{MT93}) for
some constant $c>0$. However $dx$ is an invariant measure of $R_1$,
which is infinite. This contradiction yields that $dim(\II)=0$.

\item{(iii)} If $M$ is stochastically complete and the volume of
$M$ is finite, then $dim(\II)=1$ and $\II$ coincides with $\rr$ and
consequently every integrable harmonic function is constant. The
argument is as in (ii). See Example \ref{exa41} for a stochastically
complete and finite volume manifold for which $\Delta$ is not
$L^\infty$-unique and the $L^1$-Liouville property is violated.

\ndes

} \nexam

The argument in the example above leads to

\bcor\label{cor21} Let $\LL^V=\LL=\Delta +b \cdot\nabla$, i.e.,
$V=0$. Assume that $\LL$ is $L^\infty(M,dx)$-unique. Then

\bdes \item{(a)} If the diffusion $(X_t)_{0\le t<\sigma}$ generated
by $\LL$ is explosive, i.e., $\pp_x(\sigma<+\infty)>0$ for some (or
equivalently for all) $x\in M$, then every $\LL^*$-harmonic function
$h$ in $L^1(M,dx)$ is zero.

\item{(b)} If the diffusion $(X_t)_{0\le t<\sigma}$ generated by
$\LL$ is not explosive, i.e., $\pp_x(\sigma<+\infty)=0$ for all
$x\in M$, then either there is no non-zero $\LL^*$-harmonic and
integrable function, or there is one positive $dx$-integrable
$\LL^*$-harmonic function $h_0$ such that for every non-zero
$\LL^*$-harmonic function $h\in L^1_{loc}(M,dx)$, if $h\ge 0$ or
$h\in L^1(M,dx)$, then $h$ is a constant multiple of $h_0$. \ndes
\ncor

In summary if $\LL^V$ is $L^\infty$-unique, we have the
$L^1$-Liouville property stated in Theorem \ref{thm22} and Corollary
\ref{cor21}. Then the main task remained to us is to check the
$L^\infty$-uniqueness of $\LL^V$.

From now on in this paper, $L^{\infty}$ will be endowed with the
topology $\CC (L^{\infty}, L^1)$,  and the $L^\infty$-uniqueness of
operators and $C_0$-semigroups etc. on $L^{\infty}$ are always
w.r.t. $\CC (L^{\infty},  L^1)$.

\section{$L^\infty$-uniqueness of  Sturm-Liouville operator}

Consider the following Sturm-Liouville operator:
\begin{equation}\label{generator}
\LL^Vf(x) = a(x)f^{\prime\prime} + b(x)f^{\prime}-V(x)f,
\quad\forall f\in C_0^{\infty}(x_0,y_0)
\end{equation}
 $(-\infty\leq x_0 <y_0\leq
+\infty)$. Assume that the coefficients $a,\  b,\ V$ of $\LL^V$ in
(\ref{generator}) satisfy
\begin{equation}\label{ass1}
 \  a(x),\    b(x)\in L^{\infty}_{loc}(x_0, y_0)
\end{equation}

\begin{equation}\label{ass2}
  a(x)>0 \ dx{\rm-a.e.} \ ;\ \frac 1{a(x)},\  V(x) \in
  L^{\infty}_{loc}(x_0, y_0); \  V(x)\geq0;
\end{equation}
where $L^{\infty}_{loc}(x_0, y_0)$ (resp. $L^{1}_{loc}(x_0, y_0)$ )
denotes the space of real measurable functions which are essentially
bounded (resp. integrable) w.r.t. the Lebesgue measure $dx$ on any
compact sub-interval of $(x_0, y_0)$. Fix a point $c\in (x_0, y_0)$
and let
\begin{equation}\label{feller1}
s'(x)=\exp\left(-\int^x_c\frac{b(t)}{a(t)}dt\right), \  m'(x)
=\dps\frac{1}{a(x)}\exp\left(\int^x_c\frac{b(t)}{a(t)}dt\right).
\end{equation}
 Their primitives $s$ and $m$ are respectively the scale and
speed functions of Feller. Below $m$ will also denote the measure
$m'(x)dx$. It is easy to see that
$$
\<\LL^Vf,\   g\>_m=\<f,\   \LL^Vg\>_m,  \ \forall f, g\in
C^\infty_0(x_0, y_0)
$$
where $\dps \<f,\  g\>_m:=\dps\int^{y_0}_{x_0}f(x)g(x)m'(x)dx.$ For
$f\in C^\infty_0(x_0, y_0)$,  we can write $\LL^V$  in the following
form of Feller,
$$
\LL^Vf= \frac {d}{dm} \frac d{ds}f-Vf.
$$
Now regard $\LL^V$ as an operator on $L^p(m):=L^p((x_0,y_0),m)$,
$p\in [1,+\infty]$, with domain of definition  $C^\infty_0(x_0,
y_0)$. Recall that $ L^\infty(m)$ is endowed always with the
topology $\CC(L^\infty(m), L^1(m))$. Again let $(X_t)_{0\le
t<\sigma}$ be the diffusion in $(x_0,y_0)$ generated by $\LL$ with
the explosion time $\sigma$ (cf. \cite{IM}) and define $P_t^V$ by
the Feynman-Kac formula  as in (\ref{21a}). $P_t^V$ is
$m$-symmetric, and its generator $\LL^V_{(p)}$ in
$L^p(m)=L^p((x_0,y_0), m)$ extends $\LL^V$ (defined on
$C_0^\infty(x_0,y_0)$). The problem resides again in the uniqueness.

$\LL^V$ is said {\it $L^p(m)$-unique} ($1\le p\le +\infty$), if
 its closure in $L^p(m)$ coincides with $\LL^V_{(p)}$. That
 $L^p(m)$-uniqueness is equivalent to the uniqueness of
 solution $t\to u(t)$ (continuous from $\rr^+\to L^q(m)$) to the following integral version of
 Fokker Planck equation
$$
\<u(t)-u(0), \LL^V f\>_m =\int_0^t \<u(s), \LL^V f\>_m ds, \ \forall
f\in C_0^\infty(x_0,y_0),\ \forall t\ge0
$$
for every $u(0)\in L^q(m)$ given, where $q$ is the conjugate number
of $p\in [1,+\infty]$, i.e., $\dps q=\frac{p}{p-1}$. ({\it In other
words the $L^p$-uniqueness of $\LL^V$ is equivalent to the
$L^q$-uniqueness of the associated Fokker-Planck equation.}) It is
also equivalent to : for any $h\in L^q(m)$,

\bequ\label{Lq} (\<h, (\LL^V-1)f\>_m=0, \ \forall f\in
C_0^\infty(x_0,y_0))\implies h=0. \nequ See \cite{WZ2} for numerous
other characterizations.

The study of $L^2$-uniqueness of the Sturm-Liouville operators was
born with the {\it limit point--limit cycle} theory of Weil (see
\cite{RS}). In a series of pioneering works (here we mention only
\cite{fe1, fe2}) W. Feller investigated thoroughly  the different
sub-Markov generator-extensions of $\LL^V$.

The recent study is concentrated on the case where $V=0$. Wielens
\cite{[Wi]} obtained the characterization of $L^2$-uniqueness (or
equivalently the essential self-adjointness) of $\LL$. Furthermore,
Eberle \cite{[Eb]} and Djellout \cite{Dj97} have completely
characterized the $L^p$-uniqueness of $\LL$ for $1<p<\infty$. The
$L^1$-uniqueness, the $L^\infty$-uniqueness are characterized in
\cite{Wu99} and \cite{WZ2}, respectively.

In presence of the killing potential, the problem of uniqueness
becomes much more difficult, just because it is hard to obtain {a
priori} estimates about solutions of the second order ordinary
differential equation with a potential. This can be seen for an
example in the theory of Weil: $\Delta - c/x^2$ $(c>0)$ acting on
$C_0^\infty(0,+\infty)$ is $L^2((0,+\infty),dx)$-unique iff $c\ge
3/4$ (see Reed-Simon \cite{RS}). This simple example (but profound
characterization) excludes any easy integral test criteria such as
those in no killing case.

Our purpose is to find an explicit characterization of the
$L^{\infty}$-uniqueness of $(\LL^V, C^\infty_0(x_0, y_0))$.

\subsection{Main result}
The main result of this section is

\begin{thm}\label{thm1}
$(\LL^V, C^\infty_0(x_0, y_0))$ is unique in $L^{\infty}(m)$  iff
for some or equivalently all $\delta>0$
\begin{equation}\label{1.1}
\int_{c}^{y_0}\sum_{n\geq0}I_n^{V+\delta}(y)m'(y)dy=+\infty,
\end{equation}
\begin{equation}\label{1.2}
\int_{x_0}^{c}\sum_{n\geq0}J_n^{V+\delta}(y)m'(y)dy=+\infty,
\end{equation}
where for all $V\ge 0$,

$$
\aligned I_0^V(y)&=1,\ \ I_n^V(y)=\int_{c}^{y}
s'(r)dr\int_{c}^{r}m'(t)
V(t)I^V_{n-1}(t)dt, \ y\ge c;\\
J_0^V(y)&=1,\ \
J_n^V(y)=\int_{y}^{c}s'(r)dr\int_{r}^{c}m'(t)V(t)J^V_{n-1}(t)dt, \
y\le c.
\endaligned
$$

\end{thm}

\bdef \label{def31}We say that $y_0$  (resp. $x_0$) is {\it no
entrance} boundary for $\LL^V$ if (\ref{1.1}) (resp. (\ref{1.2}))
holds for some or equivalently for all $\delta>0$. \ndef

In other words the $L^\infty$-uniqueness of $\LL^V$ is equivalent to
say that $x_0, y_0$ are no entrance boundary in the sense of
Definition \ref{def31}. In the presence of the killing potential
$V\ge0$, our definition of no entrance boundary is different from
the classical one of Feller (see Ito-Mckean \cite{IM}), so it is a
new notion. The comparison is given in Corollary \ref{cor32} and
Remarks \ref{rem33}.

\brmk\label{thm1-rmk1} {\rm Denote by $I_n^V(x)$ (resp. $J_n^V(x)$)
by $I_n^V(c;x)$ (resp. $J_n^V(x;c)$). One can prove that (\ref{1.1})
and (\ref{1.2}) do not depend on $c\in (x_0,y_0)$. Its proof is
given later.
} \nrmk

\subsection{Proof of Theorem \ref{thm1}}

Throughout this section, the dual operator $(\LL^V)^*$ is taken
w.r.t. $m$, NOT w.r.t. $dx$ unlike in other places of the paper.

We begin with a series of technical lemmas similar to {\bf Lemma
4.5, Lemma 4.6, Lemma 4.7 } of \cite{WZ2}, so we omit their proofs.

\begin{lem}\label{domain}
Let $u,v\in L^1_{loc}((x_0,y_0),m)$ such that

$$
\<u, \LL^V f\>_m = \<v, f\>_m,\ \forall f\in C_0^\infty(x_0,y_0).
$$
Then

\bdes \item{(i)}  $u$ has a $C^1$-smooth $dx$-version ${\tilde u}$
such that ${\tilde u}'$ is absolutely continuous;

\item{(ii)} $g:=a \tilde u'' +b \tilde u'-V\tilde u=(1/m')
(\tilde u'/s')'-V\tilde u\in L^1_{loc}(m)$. \ndes In that case
$v=g$.
\end{lem}

\begin{lem}\label{diff}  Suppose that $h$ is $C^1(x_0,y_0)$ such that $h'$ is absolutely
continuous  and $(h'/s')'=Vm'h$.

Assume that $c_1\in (x_0, y_0)$,  $h(c_1)>0$ and $h'(c_1)>0$ (resp.
$h'(c_1)<0$). Then $h'(y)>0$ (resp. $h'(y)<0$) for $\forall y\in
(c_1, y_0)$ (resp. $\forall y\in (x_0, c_1)$).
\end{lem}

\begin{lem}\label{solution} {\rm (essentially due to Feller \cite{fe1})}
Assume that $V\ge\delta>0$, $dx-a.e.$, then there exist two strictly
positive $C^1$-functions $h_k,  k=1, 2$ on $(x_0, y_0)$ such that

(1) For $k=1,  2$,  $h_k'$ is absolutely continuous,  and $(h'_k/s'
)'=m'V h_k$,  a.e.;

(2) $h_1'>0$ and $h_2'<0$ over $(x_0, y_0)$.
\end{lem}

Our key observation is

\bprop\label{prop31} Let $h$ be any $C^1$-function on $(x_0,y_0)$
such that $h'$ is absolutely continuous and $(h'/s')'=m'Vh,
dx-a.e.$.

\bdes
\item{(a)} If $h(c)>0, h'(c)>0$ and $\int_c^{y_0}V(t)m'(t)dt>0$, then there is a positive constant
$C$ such that

\bequ\label{prop31a} h(c)\sum_{n=0}^\infty I^V_n(x)\le h(x) \le C
\sum_{n=0}^\infty I^V_n(x), \ \forall x\ge c. \nequ

\item{(b)} If $h(c)>0, h'(c)<0$ and $\int_{x_0}^cV(t)m'(t)dt>0$, then there is a positive constant $C$
such that

\bequ\label{prop31b} h(c)\sum_{n=0}^\infty J^V_n(x)\le h(x) \le C
\sum_{n=0}^\infty J^V_n(x), \ \forall x\le c. \nequ

\ndes

\nprop

\bprf (a) By Lemma \ref{diff},  $h'(r)>0$ for $r\in [c, y_0)$.
Notice that
\begin{align*}
  h(x)=&h(c)+\int^x_{c}h'(r)dr \\
      =&h(c)+
       \int^x_{c}\left\{\frac{h'(c)}{s'(c)}s'(r)
       +s'(r)\int^r_{c}m'(t)V(t)h(t)dt \right\}dr\\
       >&h(c)+\int^x_{c}s'(r)dr\int^r_{c}m'(t)V(t)h(t)dt.
      \end{align*}
 Thus  using the above inequality recursively, we easily obtain:
 $$h(x)\ge h(c)\sum_{n=0}^{+\infty}I^V_n(x),\ \forall x\ge c$$
which is the first inequality in (\ref{prop31a}).

For the second inequality in (\ref{prop31a}), letting $K(x):\
=\int_c^{x}(h'(c)/s'(c))s'(r)dr$ and fixing $c_0\in (c, y_0)$ such
that $\int_c^{c_0} V(t)m'(t)dt>0$, we can choose suitable positive
constants $C_1, \ C_2,\ C_3$ such that for all $x\ge c$,
$$
\aligned
   K(x) & \le
  C_2+C_1\int_{c_0}^{x}s'(r)dr\\
   & \le
   C_2+C_3\int_{c}^{x}s'(r)dr\int_{c}^{r}m'(t)V(t)dt=
 C_2+C_3I^V_1(x).
\endaligned
$$
Setting $C_4=h(c)+C_2$, we get for all $x\ge c$:
\begin{align*}
 h(x)=&h(c)+
       \int^x_{c}\left\{\frac{h'(c)}{s'(c)}s'(r)
 +s'(r)\int^r_{c}m'(t)V(t)h(t)dt \right\}dr\\
\leq &  C_4+C_3I_1(x)+\int_{c}^{x}
                s'(r)dr\int_{c}^{r}m'(t)V(t)h(t)dt.
\end{align*} Using it inductively we obtain for all $x\ge c$,

\begin{align*}
h(x)\le &  C_4+  (C_3+C_4)I_1(x)+C_3I_2(x)\\
             & \ \ \ +\int_{c}^{x}s'(r_1)dr_1\int_c^{r_1} m'(t_1)V(t_1)dt_1
               \int^{t_1}_{c}s'(r_2)dr_2\int^{r_2}_{c}m'(t_2)V(t_2)h(t_2)dt_2\\
      &  \cdots\cdots\cdots\cdots\cdots\cdots\cdots\\
       \leq & C_4+(C_3+C_4)\sum_{n=1}^{+\infty}I^V_n(x).
\end{align*}
where the second inequality in (\ref{prop31a}) follows.

(b) Similar to part(a). \nprf

Let us now to the

\bprf[Proof of Remarks \ref{thm1-rmk1}] We prove here the no
entrance property of $y_0$ does not depend on $c$. Denote $I_n^V(x)$
by $I_n^V(c;x)$ to emphasize the role of $c$. Let $x_0<c<c_1<y_0$.
By Feller's lemma \ref{solution}, there is a strictly increasing
positive $C^1$-function $h=h_1$ on $(x_0,y_0)$ such that $h'$ is
absolutely continuous and
$$
(h'/s')' = (V+\delta)h m'
$$
a.e. on $(x_0,y_0)$. Hence $h'(x)>0$ over $(x_0,y_0)$. By
Proposition \ref{prop31}, there is a constant $C>0$ such that for
all $x\ge c_1$,

$$
h(c) \sum_{n=0}^\infty I_n^{V+\delta}(c;x) \le h(x) \le C
\sum_{n=0}^\infty I_n^{V+\delta}(c_1;x).
$$
That completes the proof.
 \nprf

\begin{proof}[Proof of Theorem \ref{thm1}]  Since the constant
$\lambda_0$ defined in (\ref{radius}) is non-positive, according to
Theorem \ref{thm21} (or more precisely \cite[Theorem 2.1]{WZ2}), the
$L^\infty(m)$-uniqueness of $\LL^V$ is equivalent to : for some or
equivalently for all $\delta>0$, if $h\in L^1(m)$ such that
$$
\<h, (\LL^V-\delta)f\>_m=\<h, \LL^{V+\delta}f\>_m=0,\ \forall f\in
C_0^\infty(x_0,y_0)
$$
then $h=0$. By Lemma \ref{domain}, for such $h$, we may assume that
$h\in C^1(x_0,y_0)$ and $h'$ is absolutely continuous and
 \bequ\label{equation}(h'/s' )'=m' (V+\delta) h.\nequ

{\bf Part ``if":} Assume (\ref{1.1}) and (\ref{1.2}) hold for some
$\delta>0$. Suppose in contrary that $0\ne h\in L^1(m)$ is a
solution of (\ref{equation}). We can assume that $h>0$ on some
interval $[x_1, y_1]\subset (x_0, y_0)$ where $x_1<y_1$. Notice that
$h'\not\equiv 0$ on $(x_1, y_1)$ by (\ref{equation}).

\noindent {\bf Case (i):} $h'(c_1)>0$ for some $c_1\in (x_1, y_1)$.
 We
obtain from Proposition \ref{prop31}(a):
 $$\int_{c_1}^{y_0}h(y)m'(y)dy\ge h(c_1)\int_{c_1}^{y_0}\sum_{n=0}^{+\infty}I^{V+\delta}_n(y)m'(y)dy=+\infty; $$
which is a contradiction with the assumption that $h\in L^1(m)$.

\noindent {\bf Case (ii)}:  $h'(c_1)<0$ for some $c_1\in (x_1,
y_1)$. By Proposition \ref{prop31}(b), we have

$$\int_{x_0}^{c_1}m'(y)h(y)dy\ge h(c_1)\int_{x_0}^{c_1}\sum_{n=0}^{+\infty}J^{V+\delta}_n(y)m'(y)dy=+\infty.$$

{\bf Part ``only if":} Let us prove that  (\ref{1.2}) holds for all
$\delta>0$. Indeed assume in contrary that for some $\delta>0$,
$$
\int^c_{x_0}m'(y)\sum_{n=0}^{+\infty} J_n^{V+\delta}(y)dy<+\infty.
$$
In particular $\int_{x_0}^c m'(y)dy<\infty.$ Consider a solution $h$
of (\ref{equation}) such that $h>0$ and $h'<0$ over $(x_0, y_0)$,
whose existence is assured by Feller's Lemma \ref{solution}. We
shall prove that $h\in L^1(m)$.

{\bf (1) Integrability near $y_0$:} Let $c\in (x_0,y_0)$. For $y\in
(c,y_0)$ we have
$$ 0\geq  h'(y)/s'(y)= h'(c)/s'(c)+\int_c^ym'(t)h(t)(\delta+V(t))dt$$
which implies that $\dps \delta\int_c^{y_0}m'(t)h(t)dt\leq
-h'(c)/s'(c)<+\infty$.

{\bf (2) Integrability near $x_0$:} By Proposition \ref{prop31}(b),
$$
\int_{x_0}^c m'(t)h(t)dt\le C \int_{x_0}^c \sum_{n=0}^{+\infty}
J_n^{V+\delta}(y) m'(y)dy <+\infty.
$$
That completes the proof of the necessity of (\ref{1.2}).  For the
necessity of (\ref{1.1}) for all $\delta>0$, the proof is similar :
The only difference is to use a positive and increasing solution $h$
of (\ref{equation}) (whose existence is guaranteed by Lemma
\ref{solution}).
\end{proof}

\subsection{Several corollaries}

\begin{lem}\label{cor31-lem}
     Assume that $V(x)=0,\ \forall x\in(x_0, y_0)$. The point $y_0$
     is no entrance boundary, i.e.
        (\ref{1.1}) holds iff
          \begin{equation}\label{entr1}
           \int_{c}^{y_0}m'(y)dy\int_{c}^{y}s'(x)dx=+\infty;
          \end{equation}
       and $x_0$ is no entrance boundary, i.e. (\ref{1.2}) holds iff
           \begin{equation}\label{entr2}
            \int_{x_0}^{c}m'(y)dy\int_{y}^{c}s'(x)dx=+\infty.
             \end{equation}
\end{lem}

\begin{proof} We prove here only the equivalence between (\ref{1.1})
and (\ref{entr1}).

$(\ref{entr1})\implies (\ref{1.1}).$ Let $I_n=I_n^{V+1}$ with $V=0$.
This implication is obvious because for some $c_1\in (c,y_0)$,

$$
\int_c^{y_0} I_1(y) m'(y)dy \ge  \int_c^{y_0} m'(y)dy \int_{c_1}^y
s'(x)dx \cdot \int_c^{c_1} m'(y) dy=+\infty.
$$

$(\ref{1.1})\implies (\ref{entr1}).$ If
$m(y_0):=m([c,y_0))=+\infty$, then both (\ref{1.1}) and
(\ref{entr1}) hold true. Assume then $m(y_0)<+\infty$.  We  have
        $$\aligned
        \int_{c}^{y_0}I_n(y)m'(y)dy&\leq
        m(y_0)\int_{c}^{y_0}m'(t_1)dt_1\int_{c}^{t_1}s'(r_1)dr_1\cdots\int_{c}^{r_{n-1}}m'(t_n)dt_n\int_{c}^{t_n}s'(r_n)dr_n\\
        &\le  m(y_0)\frac{1}{n!}\left(\int_{c}^{y_0}m'(t)dt\int_{c}^{t}s'(r)dr\right)^n,\\
        \endaligned$$
hence $$ \int_{c}^{y_0}\sum_{n\geq0}I_n(y)m'(y)dy\leq
m(y_0)\exp{\left(\int_{c}^{y_0}m'(y)dy\int_{c}^{y}s'(x)dx\right)}.
$$ Then (\ref{entr1}) follows immediately from (\ref {1.1}).
\end{proof}

From the lemma above we get immediately from Theorem \ref{thm1}

\begin{cor}{\rm (\cite[Theorem 4.1]{WZ2}\label{cor31})}
$(\LL, C^\infty_0(x_0, y_0))$ is unique in $L^{\infty}(m)$  if and
only if
          \begin{equation}\label{wu1}
                \int_{c}^{y_0}m'(y)dy\int_{c}^{y}s'(x)dx=+\infty;
          \end{equation} and
          \begin{equation}\label{wu2}
            \int_{x_0}^{c}m'(y)dy\int_{y}^{c}s'(x)dx=+\infty;
            \end{equation} hold.
\end{cor}

\bcor\label{cor32} If

\bequ\label{cor32a} \int_c^{y_0} (1+V(t))m'(t) dt \int_c^t
s'(r)dr<+\infty \nequ then $y_0$ is entrance boundary for $\LL^V$.
Similarly if

\bequ\label{cor32b} \int_{x_0}^c (1+V(t))m'(t) dt \int_{t}^c
s'(r)dr<+\infty \nequ then $x_0$ is  entrance boundary for $\LL^V$.

\ncor

\bprf This is obtained by the same proof as that of Lemma
\ref{cor31-lem}. \nprf

\brmk\label{rem33} {\rm In the theory of Feller (see \cite{IM}),
(\ref{cor32a}) and (\ref{cor32b}) are used for the definition of
entrance boundary of $y_0$ and $x_0$ for $\LL^V$. So our definition
of entrance boundary for $\LL^V$ is equivalent to his one if $V=0$
by Corollary \ref{cor31}, but strictly weaker in the presence of a
zero potential $V\ge 0$. For example when $c\in (0,2)$, $0$ is
entrance boundary for $d^2/dx^2 - c/x^2$ on $(0,+\infty)$ in our
sense, but it is not in the sense of Feller's (\ref{cor32b}), see
Example \ref{exam1}.

}  \nrmk

We now turn to

\begin{thm}\label{thm1-cor2} (comparison principle) Let
$$
\LL_kf(x)= a_k(x) f''(x) + b_k(x) f'-V_k(x)f,  \ \forall f\in
C_0^{\infty}(x_0, y_0)
$$
where $(a_k,\   b_k,\  V_k),\   k=1, 2$ satisfy (\ref{ass1}) and
(\ref{ass2}). Assume that for some $c\in (x_0,y_0)$,

\begin{equation}\label{con}
a_1(x)\ge  a_2(x), V_2(x)\geq V_1(x), \end{equation} \bdes
\item{(a).} If \bequ\label{comp2}
 \frac {b_1(x)}{a_1(x)} \le \frac{ b_2(x)}{a_2(x)}, \ x\ge c,\\
\nequ and $y_0$ is no entrance boundary for $\LL_1$, so it is for
$\LL_2$.

\item{(b).} If \bequ\label{comp3}
 \frac {b_1(x)}{a_1(x)} \ge \frac{ b_2(x)}{a_2(x)},\ x\le c,\nequ
 and $x_0$ is no entrance boundary for $\LL_1$, so it is for $\LL_2$.
\ndes
\end{thm}
The conditions above are guided by the intuitive picture of no
entrance boundary. Assume $a_1=a_2$ and $y_0$ is no entrance
boundary for $\LL_1$. Condition $b_2\ge b_1$ means that the heat in
the second system described by $\LL_2$ goes more rapidly to the
boundary $y_0$ than in the first system, and condition $V_2\ge V_1$
means that the heat in the second system is killed more rapidly than
in the first. Then the heat from the boundary $y_0$ goes more
difficultly into the interior in the second system than in the first
one.

\begin{proof} We prove here only the implication for $y_0$.
 Let $I_{n,k} (k=1, 2; n\in \nn)$ denote
``$I_n^{V+1}$" with respect to (w.r.t.) $\LL_k \ (k=1, 2).$   By
conditions (\ref{con}) and (\ref{comp2}), we have
\begin{align}
& m_1'(t)(1+V_1(t))\leq m_2'(t)(1+V_2(t)); \\
& \exp\left\{\int_r^{y}\frac{b_1(u)}{a_1(u)}du\right\}\leq
\exp\left\{\int_r^{y}\frac{b_2(u)}{a_2(u)}du\right\}, \ y>r.
\end{align}
Letting $B_{n,k}:=\int_c^{y_0} m_k'(y)I_{n,k}(y)dy$, we have
 \begin{align*}
  B_{n,k}=&\iint\limits_{c\le r_1\le t_1\le y_0} \frac{1}{a_k(t_1)}
e^{\int_{r_1}^{t_1}\frac{b_k(u)}{a_k(u)}du}dr_1\,dt_1\iint\limits_{c\le
r_2\le t_2\le r_1}
\frac{1}{a_k(t_2)}e^{\int_{r_2}^{t_2}\frac{b_k(u)}{a_k(u)}du}\left(V_k(t_2)+1\right)dr_2\,dt_2\\
&\cdots\iint\limits_{c\le r_n\le t_n\le r_{n-1}}
\frac{1}{a_k(t_n)}e^{\int_{r_n}^{t_n}\frac{b_k(u)}{a_k(u)}du}\left(V_k(t_n)+1\right)dr_ndt_n\int_0^{r_n}
(V_k(t_{n+1})+1)m_k'(t_{n+1}) dt_{n+1}.
\end{align*}
Thus $\forall n$, $B_{n,1}\leq B_{n,2}$. Hence the conclusion
follows.
\end{proof}

\subsection{Dirichlet or Neumann boundary problem}
Consider the Sturm-Liouville operator $\LL^V$ on $[x_0,y_0)$ where
$x_0\in \rr$ and $x_0<y_0\le +\infty$, where $a,b,V$ satisfy always
(\ref{ass1}) and (\ref{ass2}) on $[x_0,y_0)$ (instead of
$(x_0,y_0)$). Consider
$$
\DD_D:=\{f\in C_0^\infty[x_0,y_0);\ f(x_0)=0\}
$$
and
$$
\DD_N:=\{f\in C_0^\infty[x_0,y_0);\ f'(x_0)=0\}.
$$
Denote by $\LL^V_D$ (resp. $\LL^V_N$) the operator with domain of
definition $\DD_D$ (resp. $\DD_N$), which corresponds to the
Dirichlet boundary (resp. Neumann) boundary condition at $x_0$. One
can define $\LL^V_D$ and $\LL^V_N$ similarly on $(x_0,y_0]$ where
$-\infty\le x_0<y_0<+\infty$.

With exactly the same proof we have

\bthm \label{thmA1} $\LL^V_D$ (or $\LL^V_N$) is
$L^\infty([x_0,y_0),m)$-unique iff $y_0$ is no entrance boundary.
$\LL^V_D$ (or $\LL^V_N$) is $L^\infty((x_0,y_0],m)$-unique iff $x_0$
is no entrance boundary.  \nthm

\subsection{About $L^p(m)$-uniqueness of $\LL^V$}

\bprop\label{prop3.14} (\cite{WYZ}) $\LL^V$ is $L^1(m)$-unique iff
$$
\aligned &I_1^{V+1}(y_0)=\int_c^{y_0} s'(r) dr \int_c^r m'(t) (1+
V(t))
dt=+\infty;\\
&J_1^{V+1}(x_0)=\int^c_{x_0} s'(r) dr \int_r^c m'(t) (1+ V(t))
dt=+\infty.
\endaligned
$$ \nprop

With exactly the same proof as that of Theorem \ref{thm1}, we have
by Proposition \ref{prop31},

\bprop\label{prop3.15} Let $p\in (1,+\infty)$. $\LL^V$ is
$L^p(m)$-unique iff for some or all $\delta>0$,

\bequ\label{prop3.15a} \int_c^{y_0}\left(\sum_{n=0}^\infty
I^{V+\delta}_n(y) \right)^q m'(y) dy =+\infty;\nequ

\bequ\label{prop3.15b} \int_{x_0}^c\left(\sum_{n=0}^\infty
J^{V+\delta}_n(y) \right)^q m'(y) dy =+\infty. \nequ \nprop

\bdef\label{def32} {\rm If the condition (\ref{prop3.15a}) (resp.
(\ref{prop3.15b})) is verified, $y_0$ (resp. $x_0$) will be called
{\it $L^q(m)$-no entrance boundary} for $\LL^V$. } \ndef

So the no entrance boundary in Definition \ref{def31} is $L^1(m)$-no
entrance boundary. If $V=0$, a much easier criterion is available :
 \bprop\label{prop3.16} {\rm (due to Eberle \cite{[Eb]} and
Djellout \cite{Dj97})} Assume that $V=0$. $\LL$ is $L^p(m)$-unique
iff

          \bequ\label{prop3.16a}
           \int_{c}^{y_0}\left(\int_{c}^{y}s'(x)dx\right)^qm'(y)dy=+\infty;
          \nequ
        and

\bequ\label{prop3.16b}
            \int_{x_0}^{c}\left(\int_{y}^{c}s'(x)dx\right)^qm'(y)dy=+\infty.
            \nequ
\nprop

\subsection{Several examples}
For applications of the comparison principle in Theorem
\ref{thm1-cor2}, we should have some standard examples.

\bexam\label{exam1} {\bf (combination of Weil's example and Bessel's
diffusion)} {\rm Let
$$
(x_0,\ y_0)=(0,+\infty),\ \LL^V
f=f''+\frac{\gamma}{x}f'-\frac{c}{x^2},\ c\geq0,\gamma\in \rr.
$$
When $\gamma=0$, this is Weil's example mentioned before, and when
$c=0$, it is the Bessel's process with dimension $\gamma+1$. For
this example $s'(x)=x^{-\gamma},\ m'(x)=x^\gamma$ and $V(x)=c/x^2$.
$+\infty$ is no entrance boundary for $\LL f=
f''+\frac{\gamma}{x}f'$, so for $\LL^V$ by the comparison principle
in Theorem \ref{thm1-cor2}.  Furthermore condition (\ref{prop3.16a})
is verified for $y_0=+\infty$, so does (\ref{prop3.15a}).

 One decreasing solution for
$(\LL^V)^*h=0$ is given by
$$
h_c(x)= x^{\alpha},\
\alpha=\frac{-(\gamma-1)-\sqrt{(\gamma-1)^2+4c}}{2}.
$$
We have \bdes
\item{(a).} $(\LL,\ C_0^\infty(0,+\infty))$ is $L^1-$unique if and
only if $c>0$ or $c=0$ and $\gamma\geq1$, by Proposition
\ref{prop3.14}.

\item{(b)} If $c=0$ and $p\in (1, +\infty]$, $(\LL,\ C_0^\infty(0,+\infty))$ is
$L^p(m)-$unique if and only if $\gamma\leq-1$ or $\gamma\ge 2p-1$ by
Proposition \ref{prop3.16}.
\item{(c).} Let $p\in(1,+\infty]$ and $c>0$. $(\LL,\ C_0^\infty(0,+\infty))$ is
$L^p(m)-$unique if and only if $\alpha q+\gamma\leq -1 $ (where
$\dps \alpha=\frac{-(\gamma-1)-\sqrt{(\gamma-1)^2+4c}}{2},\
\frac{1}{p}+\frac{1}{q}=1$), or equivalently
$$
c\ge c_{cr}(q,\gamma):= \frac{(\gamma+1)^2}{q^2} -
\frac{\gamma^2-1}{q}.
$$
When $p=2, \gamma=0$, we find Weil's critical value $3/4$.

\item{(d)} If $\gamma=0$ and $p\in(1,+\infty]$, $(\LL,\
C_0^\infty(0,+\infty))$ is $L^p(dx)-$unique if and only if
$c\geq\frac{1}{q^2}+\frac{1}{q}$. This is a particular case of (c).
 \ndes

\bprf[Proof of part (c)] If $\alpha q+\gamma\leq -1 $, $\int_0^1
h_c^q (x) m'(x)dx=+\infty$. By Proposition \ref{prop31},
$\sum_{n=0}^\infty J_n^V\notin L^q((0,1], m'(x)dx)$, hence
$\sum_{n=0}^\infty J_n^{V+1}\notin L^q((0,1], m'(x)dx)$ (for
$J_n^{V+1}\ge J_n^V$). Thus by Theorem \ref{thm1} and Proposition
\ref{prop3.15}, $\LL^V$ is $L^p(m)$-unique.

 If $\alpha q+\gamma > -1$,  $\int_0^1 h_{c(1+\vep)}^{q}
(x) m'(x)dx<+\infty$ for some small $\vep>0$. By Proposition
\ref{prop31}, $\sum_{n=0}^\infty J_n^{(1+\vep)V}\in L^{q}((0,1],
m'(x)dx)$. But for $x\in (0,1]$, as $V+\delta =\frac c{x^2} + \delta
\le \frac {c(1+\vep)}{x^2}$ for $\delta\in (0,\vep)$, we have
$J_n^{V+\delta}\le J_n^{(1+\vep)V}$. Therefore $\sum_{n=0}^\infty
J_n^{V+\delta}\in L^q((0,1], m)$, $\LL^V$ is not $L^p(m)$-unique by
Theorem \ref{thm1} and Proposition \ref{prop3.15}. \nprf } \nexam

\bexam{\rm Let $(x_0,\ y_0)=(0,+\infty)$, $c\geq0,\gamma,\kappa\in
\rr, \kappa\ne 0$ and
$$\LL^V f=x^{\kappa}\left(f''+\frac{\gamma}{x}f'-\frac{c}{x^2}\right),\  f\in C_0^\infty(0,+\infty).$$
We have $s'(x)=x^{-\gamma}$, $m'(x)=x^{-\kappa+\gamma}$. Let $h_c$
be given as in the previous example, which is again
$(\LL^V)^*$-harmonic function. By the same proof as above, we have :
$0$ is {\it $L^q(m)$-no entrance boundary} iff $\alpha q +\gamma
-\kappa\le -1$, where $\dps
\alpha=\frac{-(\gamma-1)-\sqrt{(\gamma-1)^2+4c}}{2}$ and $q\in
[1,+\infty)$. }\nexam

\bexam\label{exa33}{\rm Let $(x_0,y_0)=\rr$, $a(x)=1$ and
$b(x)=\gamma (|x|^\alpha)'$ for $|x|>1$ and continuous on $\rr$, and
$V(x)=c |x|^\beta$ for $|x|>1$ and continuous and nonnegative on
$\rr$. Here $\gamma \in \rr$ and $\alpha, \beta, c\ge0$. In this
example $m'(x)= e^{\gamma |x|^\alpha}$, $s'(x)=e^{-\gamma
|x|^\alpha}$ for $|x|>1$ (for simplicity we have forgotten a
constant factor in $m'$ and $s'$, which plays no role in our
history). The operator $\LL^V$ is given by
$$
\LL^V f = f'' + \gamma \alpha {\rm sgn}(x)  |x|^{\alpha-1} f' -
c|x|^\beta, \ |x|>1.
$$
{\bf 1).} For all $1<p<\infty$, $\LL^V$ is $L^p(m)$-unique  by
applying Proposition \ref{prop3.16} to the case $V=0$ and then
Proposition \ref{prop3.15}.

\vskip10pt\noindent {\bf 2).}  $\LL^V$ is $L^1(m)$-unique iff
$\gamma\le 0$ or ($\gamma>0$ and $1_{c>0}\beta\ge \alpha-2$), by
Proposition \ref{prop3.14}.

\vskip10pt\noindent {\bf 3)} $\LL$ is $L^\infty(m)$-unique iff
$\gamma\ge 0$ or ``$\gamma<0$ and $\alpha \le 2$" (by \cite[Example
4.10]{WZ2}). In such case $\LL^V$ is $L^\infty(m)$-unique by the
comparison principle in Theorem \ref{thm1-cor2}.

Let $\gamma<0$ and $\alpha>2$ below. Then $\LL$ is not
$L^\infty(m)$-unique, and our purpose is to find the critical
potential $V=c|x|^\beta$ so that $\LL^V$ is $L^\infty(m)$-unique or
equivalently $+\infty$ is no entrance boundary.

{\bf Claim : }{\it  Let $\gamma<0$ and $\alpha>2$ and
$\beta=\alpha-2$. Set $c_{cr}=|\gamma| \alpha (\alpha-2)$. If
$c>c_{cr}$, $\LL^V$ is $L^\infty(m)$-unique; if $c<c_{cr}$,  $\LL^V$
is not $L^\infty(m)$-unique. }

By the symmetry we have only to regard if $+\infty$ is no entrance
boundary. For two positive functions $f,g$, we write $f\sim g$ (at
$+\infty$), if $\lim_{x\to +\infty} f(x)/g(x)=1$; and $f\propto g$
(at $+\infty$), if there are two positive constants $C_1,C_2$ such
that $C_1 f(x)\le g(x)\le C_2 f(x)$ for all $x$ large enough (say
$x\gg 1$).

To prove the claim, consider $h=s(\log s)^{\kappa}$, where $s(1)=e$
and $\kappa>0$. We have $\LL h =h \tilde V$ where
$$
\tilde V = \frac{s'^2}{s^2}\left(\frac \kappa {\log s} +
\frac{\kappa(\kappa-1)}{(\log s)^2 }\right)\sim |\gamma|\kappa
\alpha^2 x^{\alpha -2}
$$
by using $\dps s(x)\sim \frac{e^{|\gamma| x^\alpha}}{|\gamma| \alpha
x^{\alpha-1}}$. Note that $\dps s(\log s)^{\kappa} m'\propto
\frac{1}{x^{\alpha-1-\alpha\kappa}}$. Then $h\in L^1([1,+\infty),
m)$ iff $\kappa< \kappa_0=(\alpha-2)/\alpha$. For $\kappa=\kappa_0$,
$\tilde V\sim c_{cr} x^{\alpha -2}$. Now one can conclude the claim
by means of Proposition \ref{prop31} and Theorem \ref{thm1}.

 } \nexam

\section{Riemannian manifold case: comparison with one-dimensional
case}
 In this section, let  $(M, g)$ be a connected oriented
 non-compact
Riemannian manifold of dimension $d\geq $2 with  metric $g$ without
boundary, but not necessarily complete. Throughout this section we
denote by $dx$ the volume element, given in local coordinates by $
dx|_U=\sqrt{G}dx_1dx_2\cdots dx_d, \mbox{ where }\ G=det(g_{ij})$.
Let $TM$ be the tangent bundle on $M$. Let $L_{loc}^p(M)$ ($p\in
[1,+\infty]$) be the space of all real measurable  functions $f$
such that $f1_K\in L^p(M):=L^p(M,dx)$ for every compact subset $K$
of $M$. Let $H^{1, 2}(M)$ (resp. $H^{1, 2}_{loc}(M)$) be the space
of those functions $f\in L^2(M)$ (resp. $f\in L^2_{loc}(M)$) such
that $|\nabla f|\in L^2(M)$ (resp. $|\nabla f|\in L^2_{loc}(M)$)
where the gradient $\nabla f$ is taken in the distribution sense,
and $|\cdot|$ is the Riemannian metric.

Let us consider the following operator:
\begin{equation}\label{manifoldgenerator}
 \LL^Vf(x)=\Delta f(x)+ b (x)\cdot\nabla f(x)
-V(x)f(x),\ f\in C_0^\infty(M)
\end{equation}
where $\Delta, \nabla$ are respectively the Laplace-Beltrami
operator and the gradient on $M$, and $b $ is a locally Lipschitzian
vector field, $0\le V\in L^\infty_{loc}(M)$ (assumed throughout this
section). We write $\LL$ instead of $\LL^V$ if $V=0$.

 Now regard $\LL^V$ as an operator
 on $L^\infty(M)$ which is endowed with the topology $\CC \left(L^{\infty}(M),\ L^1(M)\right)$, with domain of definition
$C_0^{\infty}\left(M\right)$. Our purpose is to find some sharp
 sufficient condition for the $L^\infty$-uniqueness of $\left(\LL^V,\
C_0^{\infty}\left(M\right)\right)$.

\vskip10pt
 {\bf{Assumption (A)}}
\bdes
 \item{\bf(1) } $\rho:\ M\longrightarrow[x_0, y_0)$ is surjective, where $0\leq
x_0<y_0\leq +\infty$, such that $\rho^{-1}\left([x_0,\  l]\right)$
is compact subset for all $l\in[x_0, y_0)$, and there is some $c\in
[x_0,y_0)$ such that $\rho$ is $C^2$-smooth and $|\nabla \rho|>0$ on
$[\rho>c]$;

 \item{\bf (2)} there  exist $\alpha (r),\  \beta(r),\  q(r) \in L_{loc}^{\infty}\left([x_0,
y_0),dr\right)$, $q(r)\geq0, \ \alpha>0,\ 1/\alpha (r)\in
L^\infty_{loc}([x_0,y_0),dr)$ and $c\in[x_0,\ y_0)$ such that
$dx$-a.e. on $[\rho>c]$, \ndes

\bequ \label{22ass2} |\nabla \rho(x)|^2\leq \alpha(\rho(x)); \nequ

\begin{equation}\label{2ass1}
\LL\rho (x)\geq \frac{\beta(\rho(x))}{\alpha(\rho\left(x\right))}
|\nabla \rho(x)|^2;
\end{equation}

\begin{equation}
\label{2ass2} V(x)\geq q(\rho (x)).
\end{equation}

\begin{thm}\label{mainthm}
Under the assumption {\bf (A)}, if $y_0$ is no entrance boundary for
$\LL^{1, q}=\alpha(r)\frac{d^2}{dr^2}+\beta(r)\frac{d}{dr}-q(r)$,
then $\left(\LL^V, C_0^\infty(M)\right)$ is
$L^{\infty}(M,dx)$-unique.
\end{thm}

\brmk{\rm Our assumption {\bf (A)} is inspired from the comparison
theorems in the theory of stochastic differential equations
(\cite[Chap. VI, Sections 4 and 5]{IW}).  Assume that
$|\nabla\rho|^2=\alpha(\rho)$. Let $(X_t)_{0\le t<\sigma}$ be the
diffusion generated by $\LL$ and $\eta_t$ by
$\alpha(r)\frac{d^2}{dr^2}+\beta(r)\frac{d}{dr}$ with
$\rho(X_0)=\eta_0>c$. If (\ref{22ass2}) holds, one can realize $X_t$
and $\eta_t$ on the same probability space so that $\rho(X_t)\ge
\eta_t$ before returning to $c$. In other words  $\rho(X_t)$ goes to
$y_0$ (i.e., $X_t$ goes to infinity) more rapidly than $\eta_t$.
Then under  the assumption {\bf (A)}, if $y_0$ is no entrance
boundary for $\LL^{1,q}$,  the heat from ``boundary" of $M$ is again
more difficult to enter into $M$ : it should be ``no entrance
boundary". The result above justifies this intuition. }\nrmk

Let us begin with a Kato type inequality.
\begin{lem}\label{kato}
If $u\in L^1(M, dx)$ satisfies
\begin{equation}\label{2equation}
\int_M u(\LL^V-1)fdx=0, \quad\forall f \in C_0^{\infty}(M)
\end{equation}
Then $u\in C^1(M)$ (more precisely one version of $u$ is
$C^1$-smooth) and
\begin{equation}\label{kato2}
-\int_M \nabla |u|\cdot  \nabla fdx+\int_M b  \cdot \nabla f\cdot
|u|dx\geq \int_M(V+1)f|u|dx
\end{equation}
for all positive, compactly supported functions $f\in
H_{loc}^{1, 2}(M)$.\\
\end{lem}
\begin{proof} By \cite[Theorem 1]{BR1}, we know $u\in H_{loc}^{1,
2}(M)\cap L_{loc}^{\infty}(M).$ Using $\Delta u = div(ub )+(V+1)u$
and Sobolev's embedding theorems recursively, $u\in C^1(M)$. The
remained proof can follow word-by-word Eberle \cite[Theorem 2.5 step
2]{[Eb]} (in ``$V=0$" case), so omitted.
\end{proof}

\blem\label{lem43} If $u\in H^{1,2}_{loc}(M)$ satisfies
$$
  \<u, \LL^V f\>=0, \ \forall f\in C_0^\infty(M)
$$
and $u=0$ $dx-a.e.$ outside of some compact subset $K$ of $M$, then
$u=0$. \nlem

This is contained in the folklore of elliptic PDE, so we omit its
proof.

\begin{proof}[\it Proof of Theorem \ref{mainthm}] According to Theorem  \ref{thm21},
we have only to show that the equation
\begin{equation} \label{mainthm1}
\int_M u(\LL^V-1)fdx=0 , \quad\forall f \in C_0^{\infty}(M)
\end{equation}
 has no non-trivial
$L^1(M,dx)$ solution. Assume by absurd that there is some non-zero
$u\in L^1(M)$ satisfying (\ref{mainthm1}). By Lemma \ref{kato},
$u\in C^1(M)$ and (\ref{kato2}) holds.

 For all $r_1,r_2$ such that $x_0\leq c<r_1<r_2<y_0$, put $h(r):=\min\{r_2-r_1,
 (r_2-|r|)^+\}$ and $f:=h(\rho(x))$. Plugging such $f$ into (\ref{kato2}) we obtain:
 \bequ \label{222}
 \int_{\{r_1\leq \rho(x)\leq r_2\}}\nabla |u| \cdot \nabla
 \rho dx-\int_{\{r_1\leq \rho(x)\leq r_2\}}b  \cdot \nabla \rho\cdot
 |u|dx\geq\int_M(V+1)|u|h(\rho) dx.
 \nequ
Since $\nabla |u|\cdot \nabla \rho=div(|u|\nabla
\rho)-|u|\Delta\rho$, we have
 \begin{align*}
\int_{\{r_1\leq \rho(x)\leq r_2\}}&\nabla |u|\cdot \nabla
 \rho dx=\int_{\{r_1\leq \rho(x)\leq r_2\}}[div(|u|\nabla
\rho)-|u|\Delta\rho ]dx\\
&\stackrel{(i)}{=}\int_{\{\rho=r_2\}}|u|\cdot|\nabla
\rho|d\sigma_M-\int_{\{\rho=r_1\}}|u|\cdot|\nabla
\rho|d\sigma_M-\int_{\{r_1\leq \rho(x)\leq r_2\}}|u|\Delta\rho dx
\end{align*}
where $(i)$ follows from the Divergence Theorem, $\sigma_M$ is the
$(d-1)$-dimensional  surface  measure on the $C^2$-smooth $\{\rho=
r\}$ induced by the volume measure $dx$. Using the preceding
equality, we get:
\begin{align*}
\int_{\{\rho=r_2\}}|u|\cdot|&\nabla
\rho|d\sigma_M-\int_{\{\rho=r_1\}}|u|\cdot|\nabla
\rho|d\sigma_M
-\int_{\{r_1\leq \rho(x)\leq r_2\}}|u|\cdot|\nabla
\rho|^2\frac{\beta(\rho)}{\alpha(\rho)}dx\\
&\stackrel{(i)}{\geq}\int_{\{\rho=r_2\}}|u|\cdot|\nabla
\rho|d\sigma_M-\int_{\{\rho=r_1\}}|u|\cdot|\nabla
\rho|d\sigma_M-\int_{\{r_1\leq \rho(x)\leq r_2\}}|u|\LL\rho dx\\
&=\int_{\{r_1\leq \rho(x)\leq r_2\}}\nabla |u|\cdot \nabla
 \rho dx+\int_{\{r_1\leq \rho(x)\leq r_2\}}|u|\Delta\rho dx-\int_{\{r_1\leq \rho(x)\leq r_2\}}|u|\LL\rho
 dx\\
&\stackrel{(ii)}{\geq} \int_M (V+1)|u|h(\rho)dx\ge
\int_M\frac{V+1}{\alpha(\rho)}|u|\cdot|\nabla
\rho|^2h(\rho)dx\\
 &\stackrel{(iii)}{\geq}\int_M\frac{q(\rho)+1}{\alpha(\rho)}|u|\cdot|\nabla
\rho|^2h(\rho)dx
\end{align*}
where $(i)$ follows from condition (\ref{2ass1}), $(ii)$ follows
from  (\ref{222}) and (\ref{22ass2}), $(iii)$ follows from
(\ref{2ass2}).

 Now let
$G(r)=\int_{\{\rho(x)\leq r\}}|\nabla \rho|^2\cdot |u| dx$.  By the
Co-area formula (Federer v\cite[Theorem 3.2.12]{Federer}): for
$r_2>r_1>c$, \bequ\label{co} \int_{\{\rho(x)\in
[r_1,r_2]\}}\mid\nabla \rho(x)\mid f(x)dx=\int_{r_1}^{r_2}
dr\int_{\{\rho(x)=r\}}f(x)d\sigma_M,  \nequ $G$ is absolutely
continuous on $r\in (c,y_0)$ and
$$G'(r) = \int_{\{\rho(x)= r\}} |\nabla \rho|\cdot|u|
d\sigma_M,\ dr-a.e. \ r>c. $$ From now on we fix $G'(r)$ as the
right hand side above. By the Co-area formula we also have
 $$\int_{\{r_1\leq \rho(x)\leq
r_2\}}|u|\cdot|\nabla
\rho|^2\frac{\beta(\rho)}{\alpha(\rho)}dx=\int_{r_1}^{r_2}G'(r)\frac{\beta(r)}{\alpha(r)}dr
$$ the previous inequality
is read as :
 \begin{equation}\label{2aaa}
 G'(r_2)-G'(r_1)-\int_{r_1}^{r_2}G'(r)\frac{\beta(r)}{\alpha(r)}dr\ge
  \int_c^{y_0}\frac{q(r)+1}{\alpha(r)}G'(r)h(r)dr
 \end{equation}
for $c<r_1<r_2.$ Since
 \begin{align*}
\int_c^{y_0}\frac{q(r)+1}{\alpha(r)}G'(r)h(r)dr &=
\int_{r_1}^{r_2}\frac{q(r)+1}{\alpha(r)}(r_2-r)G'(r)dr+\int_{c}^{r_1}(r_2-r_1)G'(r)\frac{q(r)+1}{\alpha(r)}dr\\
& =
\int_{r_1}^{r_2}\frac{q(t)+1}{\alpha(t)}G'(t)\int_t^{r_2}dsdt+(r_2-r_1)\int_c^{r_1}G'(r)\frac{q(t)+1}{\alpha(t)}dt\\
&=\int_{r_1}^{r_2}ds\int_{r_1}^{s}\frac{q(t)+1}{\alpha(t)}G'(t)dt+\int_{r_1}^{r_2}ds\int_c^{r_1}G'(t)\frac{q(t)+1}{\alpha(t)}dt\\
& = \int_{r_1}^{r_2}ds\int_{c}^{s}\frac{q(t)+1}{\alpha(t)}G'(t)dt.
 \end{align*}
Substituting this into (\ref{2aaa}), we obtain
$$
-\int_{r_1}^{r_2}G'(r)\frac{\beta(r)}{\alpha(r)}dr+G'(r_2)-G'(r_1)\geq
\int_{r_1}^{r_2}ds\int_{c}^{s}\frac{q(t)+1}{\alpha(t)}G'(t)dt
$$
for $c<r_1<r_2$. It can re-written as a differential equality :
\begin{equation}\label{felform}
 \LL^{-}G:=G''(r)-\frac{\beta(r)}{\alpha(r)}G'(r)=\int_c^{r}\frac{q(t)+1}{\alpha(t)}G'(t)dt
 +H'
\end{equation}
in distribution on $(c, y_0)$, where $H:(c,y_0)\to\rr$ is
nondecreasing and right continuous. Then $G'$ admits a right
continuous version. We shall assume $G'$ is itself right continuous
below.

Consider $\tilde
m'(r)=\exp\left(\int_c^r\frac{\beta(t)}{\alpha(t)}dt\right)$,
$s'(r)=\exp\left(-\int_c^r\frac{\beta(t)}{\alpha(t)}dt\right)$,
which are respectively the derivative of the scale and speed
function associated with $\LL^-$. Then $m'(r):=\frac {\tilde
m'(r)}{\alpha(r)}, s'(r)$  are respectively the derivative of the
speed and scale function associated with $\LL^{1,q}$. Hence we can
write (\ref{felform}) in the Feller's form,
$$
(G'/\tilde m ')'\ge s'\int_c^{\cdot}\frac{q(t)+1}{\alpha(t)}G'(t)dt.
$$
Let us prove now that $\exists r_0\in (c,y_0),\   G ' (r_0)>0.$
Indeed, if in contrary $G'(r)=0, \forall r>c$, then $u=0, dx-a.e.$
on $[\rho>c]$, i.e., outside of the compact set
$K=\rho^{-1}[x_0,c]$. Thus $u=0$ on $M$ by Lemma \ref{lem43}, a
contradiction with our assumption.

 The above inequality in distribution implies that for
$dr-a.e.\ r>r_0, r_0\in(x_0, y_0)$
\begin{align*}
\frac{dG}{d\tilde m }(r) & \geq
\frac{dG}{d\tilde m }(r_0)+\int_{r_0}^rs'(u)du\int_{r_0}^{u}\frac{q(t)+1}{\alpha(t)}G'(t)dt\\
& =\frac{dG}{d\tilde m
}(r_0)+\int_{r_0}^rs'(u)du\int_{r_0}^{u}(q(t)+1) m '(t)
\frac{dG}{d\tilde m }(t)dt.
\end{align*}
Using  the above inequality by induction as in the proof of
Proposition \ref{prop31} we get
$$
\frac{G'}{\tilde m '}(y)\geq C\sum_{n=0}^{+\infty}I_n^{q+1}(y)
$$
where $I_0^{q+1}=1, \ I_n^{q+1}(y)=\int_{r_0}^ys'(r)dr\int_{r_0}^r
(q(t)+1) m '(t) I_{n-1}^{q+1}(t)dt,\  C=\frac{dG}{d\tilde m
}(r_0)>0.$
 Using Co-area formula and our assumption, we obtain:
\begin{align*}
 \int_{\{r_0\leq\rho\}}|u|dx&\geq\int_{\{r_0\leq\rho\}}\frac{|u|\cdot|\nabla \rho|^2}{\alpha(\rho)}
 dx\\
 &=\int_{r_0}^{y_0}\frac{G'(r)}{\alpha(r)}dr= \int_{r_0}^{y_0} m '(r) \frac{dG}{d\tilde m } (r)dr\\
&\geq C\sum_{n=0}^{+\infty}\int_{r_0}^{y_0} m '(r) I_n^{q+1}(r)dr.
\end{align*}
Thus $\int_{\{r_0\leq\rho\}}|u|dx=\infty$ by our assumption that
$y_0$ is no entrance boundary for $\LL^{1,q}$. This is in
contradiction with the assumption that $u\in L^1(M, dx)$.
\end{proof}

With the same proof we have the two sides' version of Theorem
\ref{mainthm} :
\begin{thm}\label{twosides}
We suppose \bdes
 \item{\bf(1) } $\rho:\ M\longrightarrow(x_0, y_0)$ is surjective, where $-\infty\leq
x_0<y_0\leq +\infty$, such that $\rho^{-1}\left([x_1,\  x_2]\right)$
is compact subset for all $x_1<x_2$ in $(x_0, y_0)$, and there are
$c_1<c_2$ in $(x_0,y_0)$ such that $\rho$ is $C^2$-smooth, $|\nabla
\rho|>0$ on $[\rho<c_1]\cup [\rho>c_2]$ ;

\item{\bf (2)} there  exist $\alpha (r),\  \beta(r),\  q(r) \in L_{loc}^{\infty}\left((x_0,
y_0), dr\right)$, $q(r)\geq0, \ \alpha>0,\ 1/\alpha (r)\in
L^\infty_{loc}(x_0,y_0)$ and $c_1<c_2$ in $(x_0,\ y_0)$ such that
$dx-a.e. $ on $[\rho<c_1]\cup [\rho>c_2]$,
 \ndes \bequ\label{cor4.16} |\nabla \rho|^2\leq
\alpha(\rho),\ \text{$dx-a.e. $ on } \ [\rho<c_1]\cup
[\rho>c_2];\nequ
 \bequ\label{cor4.17} \LL \rho \geq|\nabla \rho|^2\frac{\beta(\rho)}{\alpha(\rho)},\
 \text{$dx-a.e. $ on } \
[\rho>c_2];\nequ \bequ\label{cor4.18} \LL \rho \leq|\nabla
\rho|^2\frac{\beta(\rho)}{\alpha(\rho)},\
 \text{$dx-a.e. $ on } \
[\rho<c_1];\nequ
  \bequ\label{cor4.19}V(x)\geq q(\rho (x)),\ \text{$dx-a.e. $ on } \ [\rho<c_1]\cup
[\rho>c_2];\nequ If $x_0,y_0$ are no entrance boundaries for
$\LL^{1, q}=\alpha(r)\frac{d^2}{dr^2}+\beta(r)\frac{d}{dr}-q(r)$,
then $\left(\LL^V, C_0^\infty(M)\right)$ is $L^{\infty}(M,
dx)$-unique.
\end{thm}

\bcor Suppose that M is a Cartan-Hadamard manifold with dimension
$d\ge 2$ (i.e. complete, simply connected with non-positive
sectional curvature). Let $d(x)$ be the distance from some fixed
point $o$ to $x$. Then

\bdes
\item{(a)} $\Delta$ is $L^\infty(M,dx)$-unique. In particular the
$L^1$-Liouville property holds : every $dx$-integrable
($\Delta$-)harmonic function is constant.

\item{(b)} Assume that  $b $ verifies $b (x)\cdot \nabla d(x)\ge
-L [1+d^2(x)], \ x\ne o$ for some constant $L>0$. Then $\LL^V=\Delta
+b \cdot \nabla-V$ is $L^\infty(M,dx)$-unique. In particular the
$L^1$-Liouville property in Theorem  \ref{thm22} and Corollary
\ref{cor21}  holds true.

\item{(c)}  Assume that  $b $ verifies $b (x)\cdot \nabla d(x)\ge
-L [1+d^\alpha(x)]$ for some constants $L>0$ and $\alpha>2$. If
$V(x)\ge cd(x)^{\alpha-2}$ with $c>L \alpha (\alpha-2)$, then
$\LL^V=\Delta +b  \cdot \nabla-V$ is $L^\infty(M,dx)$-unique. In
particular the $L^1$-Liouville property  in Theorem \ref{thm22}
holds.

 \ndes
 \ncor
 \bprf {\bf (a)} The $L^\infty$-uniqueness of $\Delta$ is a particular case of part
 (b). Then the $L^1$-Liouville theorem follows  from Example
 \ref{exa21}.

 {\bf (b)} Recall that on the Cartan-Hadamard manifold, the exponential map $\exp : T_oM\to M$
is a diffeomorphism. The distance function $d(x)$ is
$C^\infty$-smooth on $M\backslash \{o\}$, and the  Laplacian
comparison theorem says that (\cite{Chee, SY})
$$
\Delta d(x)\ge \frac{d-1}{d(x)}, \ x\ne o.
$$
Hence the assumption {\bf (A)} holds   with $\rho(x)=d(x)$,
$[x_0,y_0)=\rr^+$, $\alpha(r)=1$, $q(r)=0$ and $\beta(r)=-(L+1) r^2$
(for some point $c\in \rr^+$ large enough in {\bf (A)}). By Example
\ref{exa33}, $+\infty$ is no entrance boundary for $\LL^{1,q}$, thus
$\LL^V$ is $L^\infty$-unique by Theorem \ref{mainthm}.

 {\bf (c)} Take $\rho(x)=d(x)$ as above and $\LL^{1,q}$ as follows : $[x_0,y_0)=\rr^+$,
 $\alpha(r)=1$, $q(r)=c r^\alpha$ and $\beta(r)=-(L+\vep)
 r^\alpha$, where $\vep>0$ is small enough so that $c>(L+\vep)\alpha(\alpha-2)$.
 The assumption {\bf (A)} is satisfied. Again by  Example
 \ref{exa33}, $+\infty$ is no entrance boundary for $\LL^{1,q}$, therefore $\LL^V$ is
$L^\infty$-unique by Theorem \ref{mainthm}.
 \nprf

 \bexam\label{exa41}{\bf (the first example of \cite{LS})} {\rm Let $M$ be a
compact surface with arbitrary genus. Assume the metric on $M$
around some point $o\in M$ is flat. Hence locally around $o$ we can
write the metric in polar coordinates as
$$ds^2_0=dr^2+r^2d\theta ^2$$
we choose the new metric to be
$$
ds^2=\varrho^2ds_0^2.
$$
Choose $\varrho$ to be arbitrary outside a neighborhood of $o$ (say
$r\ge\delta$), and for $r\in(0,\delta]$,
\begin{equation}
 \varrho(\theta,r)=\varrho(r)= r^{-1}(-\log r)^{-1}\left(\log(-\log
 r)\right)^{-\alpha},
 \end{equation}where $0<\delta<e^{-2}$ and $0<\alpha\leq1$ (It is assumed in \cite{LS} that $\frac12<\alpha\le 1$).
$(M\backslash \{o\}, ds^2)$ is (metrically) complete, stochastically
complete and its  volume is finite. The Green's function on $(M,
ds_0^2)$ (with the pole at $o$) $G(o,x)=f(x)$ is a positive harmonic
on $M\setminus\{o\}$ w.r.t. $ds_0^2$, then w.r.t. $ds^2$. Let
$\Delta, \nabla, |\cdot|$ be respectively the Laplacian operator,
the gradient and the Riemannian norm in the metric $ds^2$. Note
$\Delta r=\frac{1}{r\varrho(r)^2},\ |\nabla
r|=\frac{1}{\varrho(r)}$. Consider the Sturm-Liouville operator
$\dps \LL^1:
=\frac{1}{\varrho^2(r)}\frac{d^2}{d^2r}+\frac{1}{r\varrho^2(r)}\frac{d}{dr}$,
the derivative of speed function and that of scale function of $\LL$
are respectively
$m'(r)=\varrho^2(r)\exp\left(\int_1^r\frac{1}{t}dt\right)=\varrho^2(r)r,\
s'(r)=\exp\left(-\int_1^r\frac{1}{t}dt\right)=\frac{1}{r}.$ \bdes
\item{(i). } Let $\alpha\in (0, \frac12]$. Since \begin{align*}
\int_0^\delta m'(r)dr\int_r^\delta s'(t)dt&=\int_0^\delta \varrho^2(r)r\log(\frac{1}{r})dr\\
&=\int_0^\delta \frac{1}{r\log(\frac{1}{r})(\log\log(\frac{1}{r}))^{2\alpha}}dr\\
&=+\infty,
\end{align*}
 it follows that $0$ is no entrance boundary for $\LL^1$. By Theorem
\ref{mainthm}, $\left(\Delta, C_0^\infty(M\setminus\{o\})\right)$ is
$L^{\infty}(M\backslash \{o\})$-unique. Then the $L^1$-Liouville
property holds true on $(M\backslash \{o\}, ds^2)$ by Corollary
\ref{cor21}.

\item{(ii). }If $\alpha\in (\frac12,1]$. The Green function
$f(x)=G(o,x)$ with respect to the old metric  $ds_0^2$ is
$\Delta$-harmonic and $dx$-integrable as observed in \cite{LS}.
$\Delta$ can not be $L^\infty$-unique on $M\backslash \{o\}$ by
Example
 \ref{exa21} (or Corollary \ref{cor21}).  On the other hand, we have
\begin{align*}
\int_0^1m'(r)dr\int_r^1s'(t)dt&=\int_0^1\varrho^2(r)r\log(\frac{1}{r})dr<+\infty.
\end{align*}
That shows the sharpness of Theorem \ref{mainthm}.
 \ndes

 }
\nexam

 \bexam {\rm
 Let $D$ be the unit open ball centered at the origin of $\rr^d(d\ge 2)$. Let
$\LL^V:=\Delta -V(x)$ be defined on $C_0^\infty(D)$. Let
$\rho(x)=|x|$ ($|x|$ denotes the Euclidian metric). For this
example, $|\nabla\rho(x)|=1$ and $\Delta \rho(x)=\frac{d-1}{r}\ge 0
$ where $r=r(x)=|x|$. If $V(x)\ge \frac{c}{\left(1-|x|\right)^2}$,
where $c$ is a constant such that $c\ge 2$. The assumption {\bf (A)}
is satisfied for $\LL^{1,q}:=\frac{d^2}{dr^2}-\frac{c}{(1-r)^2}$. By
Example \ref{exam1}, $1$ is no entrance boundary for $\LL^{1,q}$,
then $(\LL^V,C_0^\infty(D))$ is $L^\infty$-unique by Theorem
\ref{mainthm}.

} \nexam

\bexam{\rm
 Let us consider $\LL^V:=\Delta-V$ defined on $C_0^\infty(D)$ where $D:=\rr^d\setminus{0}(d\ge 2).$
 Let $\rho(x)=r(x):=|x|$, then $|\nabla \rho(x)|=1,\Delta
 \rho(x)=\frac{d-1}{r(x)}$. If $V(x)\ge \frac{c}{r^2(x)}$ for $x$ close to $0$ (say $|x|<\delta$), where the constant $c$ satisfies
$$
c\ge d^2 - (d-1)^2+1=2d
$$
since $0$ and $+\infty$ is no entrance boundary
 for $\dps \LL^{1,q}:=\frac{d^2}{dr^2}+\frac{d-1}{r}\frac{d}{dr}-\frac{c}{r^2} 1_{(0,\delta)}(r)$ by Example
 \ref{exam1},
 $(\LL^V,C_0^\infty(D))$
 is $L^\infty$-unique by Theorem \ref{mainthm}.

 In contrary if
 $V=c/|x|^2$ with $0\le c<2d$, as $\LL^{1,q}$ is not
 $L^\infty(m)$-unique (again by Example
 \ref{exam1}), there is some $h\in L^1((0,+\infty), m)\bigcap C^1(0,+\infty)$ such
 that $h'$ is absolutely continuous and
 $\frac{d^2h}{dr^2}+\frac{d-1}{r}\frac{dh}{dr}-\frac{ch}{r^2}=h$ (such $h$ is indeed $C^\infty$-smooth).
 Let $\tilde h(x)=h(r(x))$, we see readily that $\tilde h\in L^1(D,dx)$ and
 $\Delta \tilde h-V\tilde h=\tilde h$ over $D=\rr^d\backslash
 \{0\}$.
 Thus $\Delta-V$ defined on $C_0^\infty(D)$ is not
 $L^\infty(D,dx)$-unique.
} \nexam

\end{document}